\documentclass[11pt,a4paper]{article}

\usepackage[T1]{fontenc}
\usepackage{epsfig}
\usepackage{amsmath, amsthm}
\usepackage{amsfonts,amssymb}
\usepackage[applemac]{inputenc}

\newtheorem{conj}{Conjecture}[section]
\newtheorem{theo}[conj]{Theorem}

\newtheorem{prop}[conj]{Proposition}

\newcommand{\de}{\mathrm{d}}
\newcommand{\eps}{\varepsilon}
\newcommand{\la}{\lambda}
\newcommand{\e}{\mathrm{e}}
\newcommand{\conv}{\mathrm{conv}}
\newcommand{\diam}{\mathrm{diam}}
\newcommand{\interior}{\mathrm{int}}

\renewcommand{\H}{\mathcal{H}}

\newcommand{\R}{\mathbb{R}}
\newcommand{\N}{\mathbb{N}}

\begin{document}

\title{On the analogue of the concavity of entropy power in the Brunn-Minkowski theory}

\author{Matthieu Fradelizi\footnote{Corresponding author}  and  Arnaud Marsiglietti}
\date{}
\maketitle

{\raggedright {\scriptsize  Laboratoire d'Analyse et de Math\'ematiques Appliqu\'ees, Universit\'e Paris-Est Marne-la-Vall\'ee,\\ 5 Boulevard Descartes, Champs-sur-Marne, 77454 Marne-la-Vall\'ee Cedex 2, France \\ e-mail: matthieu.fradelizi@univ-mlv.fr; \ Phone: 33 (1) 60 95 75 31 \\ e-mail: arnaud.marsiglietti@univ-mlv.fr\\~ }}
{\raggedright {\scriptsize The authors were supported in part by the Agence Nationale de la Recherche, project GeMeCoD (ANR 2011 BS01 007 01).}}

\begin{abstract}
Elaborating on the similarity between the entropy power inequality and the Brunn-Minkowski inequality, Costa and Cover conjectured in  {\it On the similarity of the entropy power inequality and the Brunn-Minkowski inequality} (IEEE Trans. Inform. Theory 30 (1984), no. 6, 837-839) the $\frac{1}{n}$-concavity of the outer parallel volume of measurable sets as an analogue of the concavity of entropy power. We investigate this conjecture and study its relationship with geometric inequalities.
\end{abstract}


\section{Introduction}

First, let us explain the origin of the conjecture of Costa and Cover. Costa and Cover \cite{Costa} noticed
the similarity between the entropy power and the Brunn-Minkowski inequalities: for every independent random vectors $X$, $Y$ in $\R^n$, with finite entropy and for every compact sets $A$ and $B$ in $\R^n$ one has
$$
N(X+Y)\ge N(X)+N(Y)\quad{\rm and}\quad |A+B|^\frac{1}{n}\ge |A|^\frac{1}{n}+|B|^\frac{1}{n},
$$
where $|\cdot|$ denote the $n$-dimensional Lebesgue measure and
$$N(X)=\frac{1}{2\pi \e}\e^{\frac{2}{n}H(X)}$$ 
denotes the entropy power of $X$. Recall that for $X$ with density $f$ the entropy of $X$ is $H(X)=-\int f\ln f$ if the integral exists and $H(X)=-\infty$ otherwise.
Applying the Brunn-Minkowski inequality to $B=\eps B_2^n$ and letting $\eps$ tend to $0$ one gets the classical isoperimetric inequality 
$$
\frac{|\partial A|}{|A|^\frac{n-1}{n}}\ge n|B_2^n|^\frac{1}{n}=\frac{|\partial B_2^n|}{|B_2^n|^\frac{n-1}{n}},
$$
where the outer Minkowski surface area is defined by 
$$
|\partial A|=\lim_{\eps \to 0}\frac{|A+\eps B_2^n|-|A|}{\eps},
$$
whenever the limit exists. 
In the same way, Costa and Cover applied the entropy power inequality to $Y=\sqrt{\eps}G$, where $G$ is a standard Gaussian random vector (the $\sqrt{\eps}$ comes from the homogeneity of entropy power $N(\sqrt{\eps}X)=\eps N(X)$). Then by letting  
$\eps$ tending to $0$ and using  de Bruijn's identity 
$$
\frac{\de}{\de t}H(X+\sqrt{t}G)=\frac{1}{2}I(X+\sqrt{t}G),
$$
which states that the Fisher information (denoted by $I$) is the derivative of the entropy along the heat semi-group,
they obtained the following "isoperimetric inequality for entropy"
$$
N(X)I(X)\ge n.
$$
Notice that this inequality is equivalent to the Log-Sobolev inequality for the Gaussian measure, see \cite{Toulouse} chapter~9.

This analogy between the results of the Information theory and the Brunn-Minkowski theory  was later extended and further explained and unified through Young's inequality by Dembo \cite{Dembo} and later on by Dembo, Cover and Thomas \cite{Cover}. Each of these theories deal with a fundamental inequality, the Brunn-Minkowski inequality for the Brunn-Minkowski theory and the entropy power inequality for the Information theory. 
The objects of each theories are fellows: to the compact sets in the Brunn-Minkowski theory correspond the random vectors in the Information theory, the Gaussian random vectors play the same role as the Euclidean balls, the entropy power $N$ corresponds to the $1/n$ power of the volume $|\cdot|^{1/n}$ and, taking logarithm, the entropy $H$ is the analogue of the logarithm of the volume $\log|\cdot|$. Hence one can conjecture that properties of one theory fit into the other theory.

Thus, Costa and Cover \cite{Costa}, as an analogue of the concavity of entropy power with added Gaussian noise, which states that
$$ t \mapsto N(X + \sqrt{t} G) $$
is a concave function (see \cite{C} and \cite{V}), formulated the following conjecture.

\begin{conj}[Costa-Cover \cite{Costa}]\label{conj}
Let $A$ be a bounded measurable set in $\mathbb{R}^n$ then the function $ t \mapsto |A + t B_2^n|^{\frac{1}{n}} $ is concave on $\mathbb{R}_+$.
\end{conj}
They also showed using the Brunn-Minkowski inequality that this conjecture holds true if $A$ is a convex set. \\

Notice that  Guleryuz, Lutwak, Yang and Zhang \cite{GLYZ} also pursued these analogies between the two theories and more recently, Bobkov and Madiman \cite{BM} established an analogue in Information theory of the Milman's reverse Brunn-Minkowski inequality.\\

In this paper, we investigate Conjecture \ref{conj} and study its relationship with known geometric inequalities. We prove that the conjecture holds true in dimension $1$ for all measurable sets and in dimension $2$ for connected sets.  In dimension $n\ge 3$, we establish that the connectivity hypothesis is not enough and that the conjecture is false in general. We then discuss additional hypotheses which ensure its validity: we conjecture that it holds true for sufficiently large $t$ and we establish it for special sets $A$. More precisely, our main results are contained in the following theorem.

\begin{theo}
Let $n\ge 1$, $A$ be a bounded measurable set. Define $V_{A}(t)=  |A + t B_2^n|$, $t\ge 0$. 
\begin{enumerate}
\item For $n=1$, the function $V_{A}$ is concave on $\R_+$.
\item For $n=2$, if $A$ is connected then $V_{A}^\frac{1}{2}$ is concave on $\R_+$. Moreover there exists $A$ not connected such that $V_{A}^\frac{1}{2}$ is not concave on $\R_+$.
\item For $n\ge 3$, if the function $\varepsilon\mapsto |\varepsilon A+B_2^n|$ is twice continuously differentiable in the neighborhood of $0$, then there exists $t_0$ such that $V_{A}^\frac{1}{n}$ is concave on $[t_0, +\infty)$. Moreover there exists $A$ connected such that $V_{A}^\frac{1}{n}$ is not concave on $\R_+$.
\end{enumerate}
\end{theo}

In the next section, we first explain some notations, then we establish analytical properties of the parallel volume and we explore relationships between this conjecture and known geometric inequalities. In the third section, we study the $\frac{1}{n}$-concavity property of the parallel volume. In the last section, we investigate further analogies between the Information theory and the Brunn-Minkowski theory.

\section{Regularity properties of the parallel volume and links with geometric inequalities}

We work in the Euclidean space $\mathbb{R}^n$, $n \geq 1$, equipped with the $\ell_2^n$ norm $|\cdot |$, whose closed unit ball is denoted by $B_2^n$ and canonical basis is $(e_1,\dots,e_n)$. We also denote $|\cdot |$ the Lebesgue measure in $\mathbb{R}^n$. 
For non-empty sets $A,B$ in $\mathbb{R}^n$ we define their \textit{Minkowski sum}
$$ A+B = \{a + b ; \, a \in A, b \in B \}.$$
We denote by $\interior(A)$, $\overline{A}$, $\partial A$, $\conv(A)$ respectively the interior, the closure, the boundary, the convex hull of the set $A$. A function $f:\R^n\to\R_+$ is $\frac{1}{n}$\textit{-concave} if $f^{\frac{1}{n}}$ is concave on its support. \\

A set $B$ is a \textit{convex body} if $B$ is a compact convex set of $\mathbb{R}^n$ with non-empty interior. If $0$ is in the interior of $B$, then the gauge associated to $B$ is the function $\|\cdot\|_B$ defined by $\|x\|_B=\inf\{t>0 ; x\in tB\}$, for every $x\in\R^n$. Let $A$ be a bounded set. For $x\in\R^n$, we set $d_B(x,A)=\inf\{\|x-y\|_B; y\in A\}$ and we simply denote $d(x,A)=d_{B_2^n}(x,A)$. We denote by $V_{A,B}$ the function defined for $t\ge0$ by 
$$
V_{A,B}(t)=|A+tB|.
$$
Notice that  
$$
 \{x ; d_B(x,A)< t\}\subset A+tB\subset \{x ; d_B(x,A)\le t\}=\overline{A}+t B.
$$
 From the continuity of $V_{A,B}$, see (\ref{federer}) below, we get $|A+tB|=|\overline{A}+tB|$ for $t>0$. Hence we may assume in the following that $A$ is compact. 
For $B=B_2^n$, we simply denote $V_A=V_{A,B_2^n}$ the \textit{(outer) parallel volume function} of $A$ defined on $\R_+$ by
$$V_A(t)=|A+tB_2^n|.$$
The outer Minkowski surface area $|\partial A|$ of $A$ may be defined using $V_A$: if the function $V_A$ admits a right derivative at $0$ then one has 
$$
(V_A)_+'(0)=\lim_{t\to 0^+}\frac{|A+tB_2^n|-|A|}{t}=|\partial A|.
$$

\subsection{Regularity properties of the parallel volume}

Let $A$ be a compact subset of $\R^n$ and $B$ be a convex body  in $\R^n$ containing $0$ in its interior. The function $d_B(\cdot,A)$ is Lipschitz, hence from Federer's co-area formula \cite{Federer},  one has 
\begin{eqnarray}\label{federer}
V_{A,B}(t)=|A+tB|=|A|+\int_0^t \H^{n-1}(\{x; d_B(x,A)=s\})\de s,
\end{eqnarray}
where $\H^{n-1}$ denotes the $(n-1)$-dimensional Hausdorff measure. Therefore the function $V_{A,B}$ is absolutely continuous on $\R_+$.

Stach\'o \cite{Stacho} proved  a better regularity for $V_{A,B}$, he proved namely that the function $V_{A,B}$ is a $n$-Kneser function, which means that for every $0<t_0\le t_1$ and every $\la\ge 1$, one has
\begin{eqnarray}\label{Kneser}
V_{A,B}(\la t_1)-V_{A,B}(\la t_0)\le \la^n(V_{A,B}(t_1)-V_{A,B}(t_0)).
\end{eqnarray}
Stach\'o deduced that for every $0<t_0<t_1$, the function 
\begin{eqnarray*}
t\mapsto V_{A,B}(t)-t^n\frac{V_{A,B}(t_1)-V_{A,B}(t_0)}{t_1^n-t_0^n}
\end{eqnarray*}
is concave on $[t_1,+\infty)$. Thus  $V_{A,B}$  admits right and left derivatives at every $t>0$, which satisfy
\begin{eqnarray}\label{derivative}
(V_{A,B})'_+(t)\le (V_{A,B})'_{-}(t)
\end{eqnarray}
and these two derivatives coincide for all $t>0$ outside a countable set. Hence the outer Minkowski surface area of $A+tB_2^n$ exists for every $t>0$ and one has
\begin{eqnarray}\label{bord}
|\partial(A+tB_2^n)|=\lim_{\eps\to 0^+}\frac{|A+tB_2^n+\eps B_2^n|-|A+tB_2^n|}{\eps}=(V_A)'_+(t).
\end{eqnarray}
In Proposition \ref{C1} below, we show that the function $V_A$ is continuously differentiable on $[\diam(A), +\infty)$.
If $A$ is convex or with sufficiently regular boundary then the equality (\ref{bord}) also holds for $t=0$. For precise statements and comparisons between the outer Minkowski surface area and other measurements of $\partial(A+tB_2^n)$, like the Hausdorff measure, see \cite{Ambrosio}.

\begin{prop}\label{continu}
Let $A$ and $B$ be compact subsets of $\R^n$ with $B$ convex, then  the function $(s,t) \mapsto |sA+tB|$ is continuous on $\R_+ \times \R_+$. Moreover
 the functions 
 $$t\mapsto |A+tB|-t^n|B|\quad \mbox{and}\quad s\mapsto |sA+B|-s^n|A|$$ 
 are non-decreasing. In particular, the function $(s,t) \mapsto |sA+tB|$ is non-decreasing in each coordinate.
\end{prop}

\begin{proof}

Let us prove the continuity. Let $0 \leq t \leq t'$. Let $r>0$ be such that $A\subset rB_2^n$ and $B\subset rB_2^n$. Then we have
$$
|A+t B|\le|A+t'B|\le |A+tB+r(t'-t)B_2^n|. 
$$
From (\ref{federer}) the function $t'\mapsto |A+tB+t'rB_2^n|$ is continuous at $0$, thus the function $t\mapsto |A+tB|$ is continuous on $\R_+$. Since for $s >0$ and $t \ge0$
$$
|sA+tB| = s^n\left|A+\frac{t}{s}B\right|
$$
then  $(s,t) \mapsto |sA+tB|$ is continuous on $\R_+^* \times \R_+$. We also have for any $s\ge0$ and $t\ge0$
$$
|tB|\le|sA+tB|\le |srB_2^n+tB|
$$
so $(s,t) \mapsto |sA+tB|$ is continuous on $\{0\} \times \R_+$. It follows that the function $(s,t) \mapsto |sA+tB|$ is continuous on $\R_+ \times \R_+$.

The monotonicity follows from (\ref{Kneser}). Indeed, The inequality (\ref{Kneser}) may be written in a different way, as follows
$$
|A+\la t_1B|-|A+\la t_0B|\le |\la A+\la t_1 B|-|\la A+\la t_0B|.
$$
Changing variables, it also means that for every $0< s_0\le s_1$ and $0< t_0\le t_1$
$$
|s_0 A+t_1B|-|s_0A+t_0B|\le |s_1A+t_1B|-|s_1A+t_0B|.
$$
Applied first to $s_1=1$ and $s_0\to 0$, and then to $t_1=1$ and $t_0\to 0$, we deduce that the functions
\begin{eqnarray*}
t\mapsto V_{A,B}(t)-t^n|B|\quad {\rm and}\quad s\mapsto |sA+B|-s^n|A|
\end{eqnarray*}
are non-decreasing. In particular, the function $(s,t) \mapsto |sA+tB|$ is non-decreasing in each coordinate.
\end{proof}

\noindent
{\bf Remark.}
If $A$ and $B$ are any compact it is not necessarily true that the function $V_{A,B}$ is non-decreasing as can be seen from the example of $A=\{0; 4\}$ and $B= [-5,-3]\cup[3,5]$.

\subsection{Links with geometric inequalities}

Let us connect the Costa-Cover conjecture with the Brunn-Minkowski inequality and the isoperimetric inequality. We first establish that the conjecture of Costa-Cover has many equivalent reformulations.

\begin{prop}\label{equiv}
Let $A$ and $B$ be compact sets in $\R^n$, with $B$ convex. The following properties are equivalent.\\
(i) $t\mapsto |A+tB|^\frac{1}{n}$ is concave on $\R_+$.\\
(ii) $s\mapsto |sA+B|^\frac{1}{n}$ is concave on $\R_+$.\\
(iii) $\lambda\mapsto |(1-\la)A+\la B|^\frac{1}{n}$ is concave on $[0,1]$.\\
(iv) $(s,t)\mapsto|sA+tB|^\frac{1}{n}$ is concave on $\R_+\times\R_+$.
\end{prop}

\begin{proof}
(iv)$\Longrightarrow$(i), (iv)$\Longrightarrow$(ii) and (iv)$\Longrightarrow$(iii) are clear.  Let us prove that
(i)$\Longrightarrow$(iv), a similar argument easily shows that (ii)$\Longrightarrow$(iv) and (iii)$\Longrightarrow$(iv).
Let $f : \R_+\to\R$ and $g : \R_+ \times\R_+\to \R$ be defined by $f(t)=|A+tB|^\frac{1}{n}$ and $g(s,t)=|sA+tB|^\frac{1}{n}$, for every $s,t\in\R_+$.  For every $t\ge 0$ and $s>0$, we have, from the homogeneity of the volume 
$$g(s,t)=s f\left(\frac{t}{s}\right).$$
Thus for every $\la\in[0,1]$, $s_1, s_2\in(0,+\infty)$ and $t_1,t_2\in\R_+$ we get
$$
g((1-\la)s_1+\la s_2,(1-\la)t_1+\la t_2))= ((1-\la)s_1+\la s_2)f\left(\frac{(1-\la)t_1+\la t_2}{(1-\la)s_1+\la s_2}\right).
$$
Using the concavity of $f$, we deduce that 
\begin{eqnarray*}
f\left(\frac{(1-\la)t_1+\la t_2}{(1-\la)s_1+\la s_2}\right)&=&f\left(\frac{(1-\la)s_1\frac{t_1}{s_1} +\la s_2\frac{t_2}{s_2}}{(1-\la)s_1+\la s_2}\right)\\
&\ge& \frac{(1-\la)s_1f\left(\frac{t_1}{s_1}\right) +\la s_2f\left(\frac{t_2}{s_2}\right)}{(1-\la)s_1+\la s_2}\\
&=& \frac{(1-\la)g(s_1,t_1)+\la g(s_2,t_2)}{(1-\la)s_1+\la s_2}.
\end{eqnarray*}
We deduce that $g$ is concave on $(\R_+^*)^2$. Moreover, $g$ is continuous on $(\R_+)^2$ by Proposition~\ref{continu}. Hence $g$ is concave on $(\R_+)^2$. 
\end{proof}

\noindent
{\bf Remark.}
Notice that if for two fixed compact sets $A$ and $B$, with $B$ convex, the assertion (iii) of Proposition~\ref{equiv} holds true then for every $\la\in[0,1]$,
$$ |(1-\lambda) A + \lambda B|^{\frac{1}{n}} \geq (1- \lambda) |A|^{\frac{1}{n}} + \lambda|B|^{\frac{1}{n}}, $$
which is the Brunn-Minkowski inequality. Hence the conjecture of Costa-Cover ((i) of Proposition~\ref{equiv}) implies the Brunn-Minkowski inequality in the case where one set is convex.\\

Let us study the connection with the isoperimetric inequality. The Costa-Cover conjecture implies that for every $t\ge 0$ and every sufficiently regular compact set $A$
$$
\frac{1}{n} \frac{|\partial A|_{n-1}}{|A|^{1-\frac{1}{n}}}= (V_A^{1/n})'_+(0) \ge (V_A^{1/n})'_+(t)\ge\lim_{t\to+\infty} (V_A^{1/n})'_+(t)=\frac{1}{n} \frac{|\partial B_2^n|_{n-1}}{|B_2^n|^{1-\frac{1}{n}}} ,
$$
which is the isoperimetric inequality.
This would give a non-increasing path from $\frac{|\partial A|_{n-1}}{|A|^{1-\frac{1}{n}}}$ to $\frac{|\partial B_2^n|_{n-1}}{|B_2^n|^{1-\frac{1}{n}}}$ through the family
$$ \left( \frac{|\partial(A + t B_2^n)|_{n-1}}{|A + t B_2^n|^{1-\frac{1}{n}}} \right)_{t \in \R_+}. $$

We may apply the same arguments for any convex body $B$ instead of $B_2^n$. Thus, the conjecture that $t\mapsto V_{A,B}(t)^{1/n}$ is concave on $\R_+$  implies the following generalized isoperimetric inequality, also known as Minkowski's first inequality proved for example in  \cite{Schneider}, 
$$
 \frac{|\partial_B A|_{n-1}}{|A|^{1-\frac{1}{n}}}\ge  \frac{|\partial_B B|_{n-1}}{|B|^{1-\frac{1}{n}}}=n|B|^\frac{1}{n}.
$$


\section{The $\frac{1}{n}$-concavity of the parallel volume}

Recall that for $t \geq 0$, $V_A(t) = |A + t B_2^n|$ and that Costa and Cover \cite{Costa} conjectured the $\frac{1}{n}$-concavity of $V_A$ on $\R_+$, for every compact $A$. They also noticed that their conjecture holds true for $A$ being convex.  Let us repeat their argument. For every $\lambda \in [0,1]$ and $t$, $s \in \mathbb{R}_+$, from the Brunn-Minkowski inequality, one obtains
\begin{eqnarray*}
|A + ((1-\lambda)t + \lambda s) B_2^n|^{\frac{1}{n}} & = & |(1-\lambda)(A + tB_2^n) + \lambda (A + sB_2^n)|^{\frac{1}{n}} \\ & \geq &  (1-\lambda)|A + tB_2^n|^{\frac{1}{n}} + \lambda|A + sB_2^n|^{\frac{1}{n}}.
\end{eqnarray*}
Notice that from the same argument we deduce that for every convex sets $A$ and $B$, the function $V_{A,B}(t)=|A+tB|$ is $\frac{1}{n}$-concave on $\R_+$.
Hence for convex sets $A$ and $B$, the properties (i)-(iv) of Proposition~\ref{equiv} holds true. In this case, the $\frac{1}{n}$-concavity of $V_{A,B}$ on $\R_+$ is equivalent to the Brunn-Minkowski inequality (and true).

\subsection{In dimension 1}

Let us prove the Costa-Cover conjecture in dimension 1.

\begin{prop}\label{theo:2.7}
Let $A$ be a compact set in $\R$ and $B$ be a convex body in $\R$, then $t \mapsto V_{A,B}(t)= |A + t B|$ is concave on $\mathbb{R}_+$.
\end{prop}

\begin{proof}

We note that in dimension $1$, for $t_0>0$, $A + t_0 B$ is a disjoint finite union of intervals. Thus, by setting $A + t_0 B$ for an arbitrary $t_0 > 0$ instead of $A$, we can assume that $A = \cup_{i=1}^N [a_i,b_i]$, with $a_i, b_i \in \R$, $N \in \mathbb{N}^*$. Thus, for $t$ sufficiently small,
$$V_{A,B}(t)= |A + t B| = \sum_{i=1}^N (b_i-a_i+|B|t) = \sum_{i=1}^N (b_i-a_i) + |B|Nt. $$
Thus $V_{A,B}$ is piecewise affine on $\mathbb{R}_+^*$. Moreover, when $t$ increases, the slope of $V_{A,B}$ is non-increasing since the number of intervals composing $A + t B$ is non-increasing. Using that $V_{A,B}$ is continuous on $\mathbb{R}_+$, we conclude that it is concave on $\mathbb{R}_+$. 
\end{proof}

\noindent
{\bf Remarks.}
For arbitrary compact sets $A$ and $B$, the function $V_{A,B}$ is not necessarily concave as can be seen from the example of $A=\{0; 4\}$ and $B= [-5,-3]\cup[3,5]$, the same example which was given in the remark after Proposition~\ref{continu} to show that the function $V_{A,B}$ is not necessarily increasing.

\subsection{In dimension 2}

We first prove the Costa-Cover conjecture for compact connected sets in dimension 2.

\begin{theo}\label{theo:2.9}

Let $A$ be a compact subset of $\R^2$. Then, $V_A : t \mapsto |A+tB_2^2|$ is $\frac{1}{2}$-concave on $\R_+$.

\end{theo}

\begin{proof}

We proceed by approximating $A$ by finite sets, hence let us first assume that $A$ is finite, $A=\{x_1,\dots, x_N\}$. For $t>0$, let $p_A(t)$ be the number of connected components of $A+tB_2^2$ and $q_A(t)$ be the genus of $A+tB_2^2$. Notice that the functions $p_A$ and $q_A$ are piecewise constants and that $V_A$ is infinitely differentiable at every $t>0$, except at those $t$'s which are equal to $\frac{|x_i-x_j|}{2}$ for some $i,j \in \{1, \dots, N\}$ or to the radius of the circumscribed circle of a triangle $(x_i, x_j, x_k)$; hence there are only a finite number of them, $t_1 < \dots < t_m$. 

We use a key result established by Fiala in the context of Riemannian manifolds, see \cite{Fiala}, first part, section 9 "vraies parall\`eles ": for every $t \in (0 ; + \infty) \setminus \{t_1,\dots,t_m\}$, 
$$ V_A''(t)\le 2 \pi (p_A(t)-q_A(t)). $$
Notice that $p_A(t)-q_A(t)$ is equal to the Euler-Poincar\'e characteristic of $A + tB_2^2$. 

Now, we consider $t_0 \in \R_+$ such that $A+t_0B_2^2$ is connected. Then for every $t \geq t_0$, $A+tB_2^2$ is connected. Hence for every $t \in (t_0;+\infty) \setminus \{t_1, \dots, t_m\}$,
\begin{eqnarray}\label{2pi}
V_A''(t) \leq 2 \pi.
\end{eqnarray}
Let us prove that $V_A$ is $\frac{1}{2}$-concave on $(t_0,+\infty)$. By the  isoperimetric inequality, we have for every $t \in (t_0;+\infty) \setminus \{t_1, \dots, t_m\}$,
$$ 4 \pi |A + t B_2^2| \leq |\partial(A + t B_2^2)|^2, $$
we write this in this form
$$ 4 \pi V_A(t) \leq V_A'(t)^2, $$
thus, using (\ref{2pi}),
$$ 2 V_A(t)V_A''(t) \leq V_A'(t)^2. $$
Hence $\left(\sqrt{V_A}\right)''(t)\le 0$. 
We conclude that $V_A$ is $\frac{1}{2}$-concave on $(t_i , t_{i+1})$, for all $i\le m-1$ and on $(t_m,+\infty)$. From (\ref{derivative}) we have $(V_A)'_-(t_i)\ge (V_A)'_+(t_i)$, thus $V_A$ is $\frac{1}{2}$-concave on $(t_0 , + \infty)$.

Let us then consider a compact connected set $A$ of $\R^2$. Let $t_0>0$. Let $(x_N)_{N \in \N^*}$ be a  dense sequence in $A$. We denote, for $N \in \N^*$, $A_N = \{ x_1, \dots, x_N \}$. There exists $N_0 \in \N^*$ such that for every $N \geq N_0$, $A_N + t_0 B_2^2$ is connected. For every $N \geq N_0$, we have shown that $V_{A_N}$ is $\frac{1}{2}$-concave on $(t_0 ; + \infty)$. Moreover the sequence $(A_N)_N\to A$ in the Hausdorff distance, thus by denoting $d_N=d_H(A_N,A)$, the Hausdorff distance, one has, for every $t>0$ 
$$
A_N+tB_2^2\subset A+tB_2^2\subset A_N+(t+d_N)B_2^2.
$$
Applying the right hand side inclusion to $t$ replaced by $t-d_N$ where $N$ satisfies $d_N<t$, we deduce
$$
A+(t-d_N)B_2^2\subset A_N+tB_2^2\subset A+tB_2^2.
$$
Hence by continuity of the function $V_A$ at the point $t$,
$$ \lim_{N \to + \infty} V_{A_N}(t) = V_A(t). $$
It follows that $\sqrt{V_A}$ is the pointwise limit of a sequence of concave functions, hence $V_A$ is $\frac{1}{2}$-concave on $(t_0 ; + \infty)$, for every $t_0>0$. We conclude that $V_A$ is $\frac{1}{2}$-concave on $\R_+$.
\end{proof}

\noindent
{\bf Remarks.}
\begin{enumerate}
\item In the proof of Theorem~\ref{theo:2.9}, from the bound $V_A''(t) \leq 2 \pi(p_A(t)-q_A(t))$ obtained for every finite set $A$ and for every $t>0$ outside a finite number of points, one deduces that for every compact subset $A$ of $\R^2$ with finite connected components $p_A$, the function $t \mapsto V_A(t) - p_A \pi t^2$ is concave on $(0;+\infty)$. From Steiner's formula one has 
$$
V_{\conv(A)}(t) =  |\conv(A)| + t|\partial(\conv(A))| + \pi t^2.
$$
If $A$ is connected, it follows that
$$V_{\conv(A)}(t)-V_A(t) =  |\conv(A)| + t|\partial(\conv(A))| + \pi t^2- V_A(t) $$  
is convex  as the sum of an affine function and a convex function. Notice that this complements the result of Kampf \cite{Kampf1} who proved that $V_{\conv(A)}(t)-V_A(t)$ tends to $0$ as $t\to +\infty$. 
\item If in Theorem~\ref{theo:2.9} we replace $B_2^2$ by an ellipsoid, $i.e.$ by $T(B_2^2)$ where $T$ is an invertible linear transformation,  then the result holds since
$$ |A+tT(B_2^2)| = |T(T^{-1}(A)+tB_2^2)| = |\det(T)||T^{-1}(A) + tB_2^2|. $$
\end{enumerate}

For a non-connected set $A$, the next proposition shows that the function $V_A$ is not necessarily $\frac{1}{n}$-concave on $\R_+$ in dimension $n\ge 2$.

\begin{prop}
Let $n \geq 2$. We set $A=B_2^n \cup \{ 2 e_1\}$. The function $V_A(t) = |A+tB_2^n|$ is not $\frac{1}{n}$-concave on $\R_+$.
\end{prop}

\begin{proof}
For every $t \in [0,\frac{1}{2})$, we have
$$ |A + t B_2^n| = |(B_2^n \cup \{ 2 e_1\}) + tB_2^n| = |B_2^n + tB_2^n| + |t B_2^n| = |B_2^n|((1+t)^n+t^n). $$
Since the $\frac{1}{n}$-power of this function is not concave (it is strictly convex),  $V_A$ is not $\frac{1}{n}$-concave on $\R_+$ for $n\ge 2$.
\end{proof}

\noindent
{\bf Remark.} This counterexample shows that the Brunn-Minkowski inequality doesn't imply the $\frac{1}{n}$-concavity of the parallel volume for non convex sets. \\

\subsection{In dimension $n\ge 3$}

We may ask if the Costa-Cover conjecture still holds for connected sets in dimension $n \ge3$. The next proposition shows that this is false: even for star-shaped body, the function $V_A$ is not necessarily $\frac{1}{n}$-concave on $\R_+$.

\begin{prop}\label{theo:2.6}

Let $n \geq 3$. We set $A=([-1,1]^3 \cup [e_1, l e_1] )\times [-1,1]^{n-3} $, where $l \ge n^4$. The function $V_A(t) = |A+tB_2^n|$ is not $\frac{1}{n}$-concave on $\R_+$.

\end{prop}

\begin{proof}
Define $C=\{0\}^3\times[-1,1]^{n-3}$. For $t\in[0, 1]$, we have
\begin{eqnarray*}
|A + t B_2^n| & = & |[-1,1]^n + t B_2^n| + |[(1+t)e_1,le_1] +C + tB_2^{n}\cap e_1^{\perp}| \\ & ~ & \qquad  \qquad + |\{le_1\} +C + t (B_2^n)^+|
\end{eqnarray*}
where
$$ (B_2^n)^+ = \{x \in B_2^n ; x_1 \geq 0 \}. $$
Using Steiner's formula for each term (see for example \cite{Schneider} p.294 for the second term), we get that for $t\in[0,1]$
$$V_A(t)=  a_0 + a_1 t + \cdots + a_nt^n, $$
with $a_0 = 2^n$, $a_1 = n2^n$ and 
$$a_2 =2^{n-3}\pi\left(\frac{2(l-1)}{(n-1)(n-2)}+n(n-1)\right).$$ 
Since $l\ge n^4$, it follows directly that 
$$
\frac{n}{n-1}V_A(0)V_A''(0)-V_A'(0)^2>0
$$ 
Hence $(V_A^{1/n})''(0)>0$, thus $V_A^{1/n}$ is not concave in the neighbourhood of $0$.
\end{proof}

We have seen that the Costa-Cover conjecture does not hold in general. We still conjecture that the following weaker form may hold.

\begin{conj}\label{conj2}

Let $A$ be a compact subset of $\mathbb{R}^n$ and $B$ be a convex body in $\R^n$. Then there exists $t_0$ such that the function $V_{A,B}(t)= |A + tB|$ is $\frac{1}{n}$-concave on $[t_0, + \infty)$.

\end{conj}

We have shown that this conjecture is true in dimension $1$ and in dimension $2$ for $B=B_2^2$. Indeed, in dimension $2$, we have seen that it is true for every compact connected set. Since for every compact subset $A$ of $\R^2$ the set $A + t B_2^2$ is connected for $t\ge \frac{1}{2}\diam(A)$, it follows that $t \mapsto |A + tB_2^2|$ is $\frac{1}{2}$-concave on $[\frac{1}{2}\diam(A), + \infty)$.\\

We prove the Conjecture~\ref{conj2} in some particular cases in dimension $n\ge 3$.

\begin{prop}\label{rouge}

Let $A$ be a compact subset of $\R^n$. Then the function $t \mapsto |A + t \conv(A)|^{1/n}$ is affine on $[n ; + \infty)$. If moreover $\partial\conv(A)\subset A$ then $t \mapsto |A + t \conv(A)|^{1/n}$ is affine on $[1 ; + \infty)$.

\end{prop}

\begin{proof}
It was noticed by Schneider \cite{Schneider75} that for every $t \geq n $,
$$ A + t\conv(A) = (1+t)\conv(A). $$
For  $t \geq n $, we get 
$$A + t\conv(A) = A+n\conv(A)+(t-n)\conv(A)=(1+ t)\conv(A). $$
We conclude that $t \mapsto |A + t \conv(A)|^\frac{1}{n}$ is affine on $[n; + \infty)$.

If moreover $\partial\conv(A)\subset A$ then for every $x \in \conv(A)$ there exists two points $y,z$ in $\partial \conv(A)$ such that $x\in[y,z]$. Say, for example, that $|x-y| \leq |x-z|$ then $u=2x-y \in [y,z]\subset \conv(A)$. Hence
$$ x = \frac{y+u}{2} \in \frac{\partial \conv(A)+\conv(A)}{2}. $$
Finally
$$ \conv(A) \subset \frac{\partial \conv(A)+\conv(A)}{2}\subset \frac{A+\conv(A)}{2}\subset\conv(A).$$
We deduce that $ A + t\conv(A) = (1+t) \conv(A), $ for every $t \geq 1$.
We conclude that $t \mapsto |A + t \conv(A)|^\frac{1}{n}$ is affine on $[1; + \infty)$.
\end{proof}

\noindent
{\bf Remark.}
More generally, Schneider introduced in \cite{Schneider75} the quantity
$$ c(A)=\inf\{t\ge 0 ; A+t\conv(A)=(1+t)\conv(A)\}. $$
Clearly  $t \mapsto |A + t \conv(A)|^\frac{1}{n}$ is affine on $[c(A); + \infty)$. The above proposition establishes that $c(A)\le n$ in general and $c(A)\le1$ if $\partial\conv(A)\subset A$. Notice that if $A\subset\R^n$ is connected then $c(A)\le n-1$, see  \cite{Schneider75}.

\begin{theo}\label{tgrand}
Let $A$ be a compact set in $\mathbb{R}^n$. If the function $\varepsilon \mapsto |\varepsilon A + B_2^n|$ is twice differentiable in a neighbourhood of $0$, with second derivative continuous at $0$, then there exists $t_0\ge0$ such that the function $V_A(t)= |A + t B_2^n|$ is $\frac{1}{n}$-concave for $t\ge t_0$. In particular this holds for $A$ being finite.
\end{theo}

\begin{proof}
Kampf proved in \cite{Kampf2}, lemma 28, that for every compact set $A$ there exists a constant $C$ which depends on $n, A$ so that for every $t \geq 1$,
$$ 0 \leq |\conv(A) + t B_2^n| - |A + t B_2^n| \leq C t^{n-3}. $$
Then, setting $\varepsilon = \frac{1}{t}$, for every $\varepsilon \in (0,1]$, one deduces
\begin{eqnarray}\label{en0} 
0 \leq |\varepsilon \conv(A) + B_2^n| - |\varepsilon A + B_2^n| \leq C \varepsilon^3. 
\end{eqnarray}
We denote $g_{\conv(A)}(\varepsilon) = |\varepsilon \conv(A) + B_2^n|$ and $g_A(\varepsilon) = |\varepsilon A + B_2^n|$, since $g_A$ is twice differentiable at $0$  it follows that
$$ g_A(0) = g_{\conv(A)}(0) \, ; \, g_A'(0) = g_{\conv(A)}'(0) \, ; \, g_A''(0) = g_{\conv(A)}''(0). $$
From Steiner's formula, we get $g_{\conv(A)}(0)= |B_2^n|$ and
\begin{eqnarray*}  g_{\conv(A)}'(0) & = & nV(\conv(A),B_2^n[n-1]), \\ g_{\conv(A)}''(0) & = & n(n-1)V(\conv(A)[2], B_2^n[n-2]).
\end{eqnarray*}
 If $\conv(A)$ is not homothetic to $B_2^n$, then from the equality case of the Alexandrov-Fenchel inequality, see \cite{Schneider}, theorem 6.6.8, page 359, we get
$$ |B_2^n| V(\conv(A)[2], B_2^n[n-2])< V(\conv(A),B_2^n[n-1])^2, $$
that is
$$ \frac{n}{n-1} g_{\conv(A)}(0) g_{\conv(A)}''(0) < g_{\conv(A)}'(0)^2. $$
Thus we deduce that0
$$ \frac{n}{n-1} g_A(0) g_A''(0) < g_A'(0)^2. $$
Since $g_A$, $g_A'$ and $g_A''$ are continuous at $0$, there exists $\varepsilon_0 > 0$ such that for every $\varepsilon \in [0, \varepsilon_0]$,
$$ \frac{n}{n-1} g_A(\varepsilon) g_A''(\varepsilon) \leq g_A'(\varepsilon)^2. $$
Hence the function $g_A$ is $\frac{1}{n}$-concave on $[0, \varepsilon_0]$. We conclude by Proposition~\ref{equiv}, setting $t_0 = \frac{1}{\varepsilon_0}$, that $t \mapsto |A + tB_2^n|$ is $\frac{1}{n}$-concave on $[t_0, +\infty)$.
If $\conv(A)$ is homothetic to $B_2^n$ then the result  follows from Proposition~\ref{rouge}. 

If $A$ is finite then the function $\varepsilon \mapsto |\varepsilon A + B_2^n|$ is analytic in a neighbourhood of $0$, see \cite{Gorbovickis}.

\end{proof}

\noindent
{\bf Remarks.}
\begin{enumerate}
\item The preceding theorem is still valid if one replaces $B_2^n$ by a convex body $B=rB_2^n+M$, for some $r>0$ and some convex body $M$ such that  its support function $h_B(u)=\max\{ <x,u>, x \in B \}$ is twice differentiable on $\R^n\setminus\{0\}$ because inequality (\ref{en0}) of \cite{Kampf2} holds with these assumptions.
\item The function $\varepsilon \mapsto |\varepsilon A + B_2^n|$ is not necessarily twice differentiable in a neighbourhood of $0$ as can be seen from the following example. In dimension 2, we consider the points $I=(1,1)$, $J=(1,0)$ and $A=I\cup J\cup\{(\cos(1/k), \sin(1/k)), k \geq 1\}$. Then, $A$ is compact but for every $t_0 \in \R_+$, the function $V_A(t)= |A + t B_2^2|$ is not twice differentiable on $(t_0, + \infty)$.
 \end{enumerate}

In fact, one can show that the function $V_A(t)= |A + t B_2^n|$ is continuously differentiable on $[\diam(A); + \infty)$.

\begin{prop}\label{C1}

Let $A$ be a compact subset of $\R^n$. Then the function $V_A(t)= |A + t B_2^n|$ is continuously differentiable on $[\diam(A); + \infty)$, the function $g_A(\eps)= |\eps A+B_2^n|$ is continuously differentiable on $(0, \frac{1}{\diam(A)}]$ and differentiable at $0$ with $g_A'(0)=nV(\conv(A),B_2^n[n-1])$.

\end{prop}

\begin{proof}

Rataj et al. in \cite{Rataj}, theorem 3.3, showed that $V_A'(t)$ exists for every $t \geq \diam(A)$, thus we have for every $t \geq \diam(A)$
$$ V_A'(t) = |\partial (A + t B_2^n)|. $$
Moreover, if $(A_N)$ is a sequence of non-empty compact subset of $\R^n$ tending in Hausdorff distance to a compact subset $A$ of $\R^n$, then by \cite{Stacho}, theorem~3, for every $t>0$ such that $V_A'(t)$ exists
$$ \lim_{N \to + \infty} |\partial (A_N + t B_2^n)| = |\partial (A + t B_2^n)|. $$
Let $t\ge \diam(A)$, we apply this result to $A_N = A + \frac{1}{N} B_2^n$. We obtain that
$$ \lim_{N \to + \infty} V_A'\left(t + \frac{1}{N}\right) = V_A'(t). $$
Hence, $V_A'$ is right continuous at $t$. Let $t>t_0>\diam(A)$, we now apply the result of Stach\'o to $ A_N = A + (t_0-\frac{1}{N}) B_2^n $. We obtain
$$ \lim_{N \to + \infty} |\partial (A_N + (t-t_0) B_2^n)| = |\partial (A + t_0 B_2^n + (t-t_0)B_2^n)| $$
that is
$$ \lim_{N \to + \infty} V_A'\left(t-\frac{1}{N}\right) = V_A'(t). $$
Hence, $V_A'$ is left continuous at $t$. We conclude that $V_A$ is continuously differentiable on $[\diam(A); + \infty)$. 

Let us denote $g_A(\eps)=|\eps A+B_2^n|$.
Since 
$$
g_A(\eps)=|\eps A+B_2^n|=\eps^nV_A\left(\frac{1}{\eps}\right)
$$
one gets that $g_A$ is continuously differentiable on $(0, \frac{1}{\diam(A)}]$. Moreover, from the inequality (\ref{en0}), valid for any compact set $A$, one deduces that $g_A$ is also  differentiable at $0$, with $g_A'(0)=nV(\conv(A),B_2^n[n-1])$.
\end{proof}

\subsubsection{A special case in dimension 3}

We have seen that for every finite subset $A$ of $\R^n$, there exists $t_0(A)$ such that the function $V_A(t)= |A + t B_2^n|$ is $\frac{1}{n}$-concave for $t\ge t_0(A)$. In dimension 3, we can give a bound on $t_0(A)$ in terms of the geometry of $A$.

In the sequel, $A$ denotes a finite subset of $\R^3$. We denote by $D_i$ a Dirichlet-Voronoi cell with respect to $A=\{x_1,\dots, x_N\} $, defined for $i \in\{1,\cdots, N\}$ by 
$$
D_i=\{x\in \R^3 ; |x-x_i| \le |x-x_j|, \forall j\in\{1,\dots, N\}\}.
$$
The following condition can be found in \cite{KK}.

\paragraph*{Condition ($\star$)}
For all faces $F$ of the polytope $\conv(A)$, and all edges $E$ of $F$, we have
$$ \forall x \in E, ~ d(x,A \cap E) = d(x,A \cap F).$$

For example, if $\conv(A)$ is simplicial, this condition holds if and only if each face of $\conv(A)$ is a triangle with only acute angles. In general, this condition holds if and only if for every face $F$ of $\conv(A)$, for every edge $[a,b]$ of $F$ and for every vertex $c$ of $F$, the angle $(ca,cb)$ is acute.

\begin{prop}

Let $A$ be a finite set in $\R^3$ satisfying the condition $\mathrm(\star)$. Then, $V_A(t)= |A + t B_2^3|$ is $\frac{1}{3}$-concave on $[t_0(A) ; + \infty)$, where
$$ t_0(A) = \min \{ t \geq \diam(A) \, ; \, D_i \subset A + tB_2^3, \, \mathrm{for \, all \, bounded} \, D_i \}. $$

\end{prop}

\begin{proof}

Kampf and Kiderlen have shown in \cite{KK} that for every $t > t_0(A)$,
$$ |\conv(A) + t B_2^3| - |A + t B_2^3| = a_0 + \sum_{p \geq 1} a_p t^{-2p+1} $$
with for all $p \geq 0$, $a_p \geq 0$. Since $V_{\conv(A)}$ is polynomial thus $V_{A}$ is twice differentiable on $(t_0(A) ; + \infty)$. It follows that for every $t > t_0(A)$,
\begin{eqnarray*}
V_{A}'(t) & = & V_{\conv(A)}'(t) + \sum_{p \geq 1}(2p - 1)a_pt^{-2p} \\ V_{A}''(t) & = & V_{\conv(A)}''(t) - \sum_{p \geq 1}2p(2p-1)a_pt^{-2p-1}.
\end{eqnarray*}
Then, for every $t > t_0(A)$,
\begin{eqnarray}\label{strategy}
V_{A}(t) \leq V_{\conv(A)}(t),\ V_{A}'(t)  \geq V_{\conv(A)}'(t) \ \mbox{and}\ V_{A}''(t)  \leq  V_{\conv(A)}''(t). 
\end{eqnarray}
The Brunn-Minkowski inequality implies that $V_{\conv(A)}$ is $\frac{1}{3}$-concave on $\R_+$. We conclude that for every $t > t_0(A)$,
$$ \frac{3}{2} V_{A}(t) V_{A}''(t) \leq \frac{3}{2} V_{\conv(A)}(t) V_{\conv(A)}''(t) \leq V_{\conv(A)}'(t)^2 \leq V_{A}'(t)^2. $$
So, $V_{A}$ is $\frac{1}{3}$-concave on $[t_0(A) ; + \infty)$.
\end{proof}

\noindent
{\bf Remarks.}
\begin{enumerate}
\item For an arbitrary compact subset $A$ of $\R^3$, if there exists a sequence $(x_N)_{N \in \N^*}$ dense in $A$ such that for every $N$, the set $A_N$ satisfies the condition $\mathrm(\star)$, where $A_N = \{ x_1, \dots, x_N \}$, and such that $t_0(A_N)$ is uniformly bounded in $N$ by a $t_0$, then the function $t \mapsto |A + t B_2^3|$ will be $\frac{1}{3}$-concave on $[t_0 ; + \infty)$.
\item In dimension $n\ge 4$, there is no hope to prove the inequalities (\ref{strategy}) because  for $A$ being  two points at distance $2$, one has for every $t\ge1$
\begin{eqnarray*}
V_A'(t)&=&n|B_2^n|t^{n-1}+2(n-1)|B_2^{n-1}|t\int_0^1(t^2-x^2)^\frac{n-3}{2}\de x\\
&<&n|B_2^n|t^{n-1}+2(n-1)|B_2^{n-1}|t^{n-2}=V_{\conv(A)}'(t).
\end{eqnarray*}
\end{enumerate}

\section{Further analogies}

In Information theory,  the Blachman-Stam inequality (\cite{B} and \cite{St}), which states that for any independent random vectors $X$ and $Y$ in $\R^n$ with non-zero Fisher information one has
$$
I(X+Y)^{-1}\ge I(X)^{-1}+I(Y)^{-1},
$$
 directly implies all previous mentioned inequalities of Information theory: the entropy power inequality   (thus the Log-Sobolev inequality for Gaussian measure) and the concavity of entropy power. This last inequality also called the "isoperimetric information inequality" may be deduced from the Blachman-Stam inequality in the same way as the "isoperimetric entropy inequality" was deduced from the entropy power inequality, by applying it to $Y=\sqrt{\eps}G$ and letting $\eps$ tend to $0$.

Let us now investigate the analogue of the Fisher information and the Blachman-Stam inequality in the Brunn-Minkowski theory. Recall de Bruijn's identity
$$
I(X)= \frac{\de}{\de t}_{|t=0}2H(X+\sqrt{t}G).
$$
 Since the entropy $H$ is the analogue of the logarithm of the volume $\log|\cdot|$, Dembo, Cover and Thomas \cite{Cover} proposed, as an analogue of the Fisher information $I$, the quantity
$$
\frac{\de}{\de\eps}_{|\eps=0}( \log|A+\eps B_2^n|)=\frac{|\partial A|}{|A|},
$$ 
for sufficiently regular compact sets $A$. Thus, in analogy with the Blachman-Stam inequality, one may wonder if for every regular compact sets $A$ and $B$ 
\begin{eqnarray}\label{DCTconj}
\frac{|A+B|}{|\partial(A+B)|}\ge \frac{|A|}{|\partial A|}+\frac{|B|}{|\partial B|}.
\end{eqnarray}
Even restricted to the case where $A$ and $B$ are convex sets, checking the validity of this inequality is not an easy task and it was conjectured by Dembo, Cover and Thomas \cite{Cover} that the inequality (\ref{DCTconj}) holds true in this particular case.
In \cite{FGM}, it was shown that this conjecture (for convex sets) holds true in dimension $2$ but is false in dimension  $n\ge 3$. In particular, it was proved that, if $n\ge 3$, there exists a convex body $K$ such that the inequality (\ref{DCTconj}) cannot be true for all $A, B\in \{K+tB_2^n ; t\ge 0\}$. It was also proved that if $B$ is a segment then there exists a convex body $A$ for which (\ref{DCTconj}) is false. 

In another direction, one may also ask if (\ref{DCTconj}) holds true for $B$ being any Euclidean ball and every compact set $A$. In this case, applying (\ref{DCTconj}) to $A$ replaced by $A+sB_2^n$ and $B=(t-s) B_2^n$, one would have, for every $0\le s\le t$, 
$$
\frac{|A+tB_2^n|}{|\partial(A+t B_2^n)|}\ge \frac{|A+sB_2^n|}{|\partial (A+sB_2^n)|}+(t-s)\frac{|B_2^n|}{|\partial B_2^n|}= \frac{|A+sB_2^n|}{|\partial(A+sB_2^n)|}+\frac{t-s}{n},
$$
with the notations given above, this would mean that 
$$
t\mapsto\frac{V_A(t)}{(V_A)'_+(t)}-\frac{t}{n}
$$
is non-decreasing on $(0,+\infty)$. This is equivalent to the $\frac{1}{n}$-concavity of $V_A$, which is the Costa-Cover conjecture. \\

\noindent
{\bf Extensions.}\\
In a work in progress \cite{M}, the second named author investigates extensions of Costa-Cover conjecture. More precisely,
 he discusses the concavity properties of the function $t\mapsto \mu(A+tB_2^n)$, where $\mu$ is a log-concave measure. 
 He also establishes functional versions of Costa-Cover conjecture.\\
 
\noindent
{\bf Acknowledgement.}\\
We thank Evgueni Abakoumov, Ludovic Goudenège and Olivier Gu\'edon for their kind remarks.

\end{document}